\documentclass[11pt,leqno]{article}
\usepackage[margin=1in]{geometry} 
\usepackage{amssymb,amsfonts,amsmath,bbm,mathrsfs,stmaryrd,mathtools}
\usepackage{xcolor}
\usepackage{url}

\usepackage{graphicx}

\usepackage{extarrows}

\usepackage[shortlabels]{enumitem}
\usepackage{tensor}

\usepackage{xr}
\usepackage[T1]{fontenc}

\usepackage[colorlinks,
linkcolor=black!75!red,
citecolor=blue,
pdftitle={},
pdfproducer={pdfLaTeX},
pdfpagemode=None,
bookmarksopen=true,
bookmarksnumbered=true,
backref=page]{hyperref}

\usepackage{tikz}
\usetikzlibrary{arrows,calc,decorations.pathreplacing,decorations.markings,decorations.shapes,intersections,shapes.geometric,through,fit,shapes.symbols,positioning,decorations.pathmorphing}

\usepackage{braket}

\usepackage[amsmath,thmmarks,hyperref]{ntheorem}
\usepackage{cleveref}

\newcommand\nthalias[1]{\AddToHook{env/#1/begin}{\crefalias{lemma}{#1}}}

\nthalias{definition}
\nthalias{example}
\nthalias{examples}
\nthalias{remark}
\nthalias{remarks}
\nthalias{convention}
\nthalias{notation}
\nthalias{construction}
\nthalias{sketch}
\nthalias{theoremN}
\nthalias{propositionN}
\nthalias{corollaryN}
\nthalias{lemma}
\nthalias{proposition}
\nthalias{corollary}
\nthalias{theorem}
\nthalias{conjecture}
\nthalias{question}
\nthalias{assumption}
\nthalias{recollection}

\creflabelformat{enumi}{#2#1#3}

\crefname{section}{Section}{Sections}
\crefformat{section}{#2Section~#1#3} 
\Crefformat{section}{#2Section~#1#3} 

\crefname{subsection}{\S}{\S\S}
\AtBeginDocument{%
  \crefformat{subsection}{#2\S#1#3}%
  \Crefformat{subsection}{#2\S#1#3}% 
}

\crefname{subsubsection}{\S}{\S\S}
\AtBeginDocument{%
  \crefformat{subsubsection}{#2\S#1#3}%
  \Crefformat{subsubsection}{#2\S#1#3}% 
}

\theoremstyle{plain}

\newtheorem{lemma}{Lemma}[section]
\newtheorem{proposition}[lemma]{Proposition}
\newtheorem{corollary}[lemma]{Corollary}
\newtheorem{theorem}[lemma]{Theorem}

\theoremstyle{plain}
\theoremnumbering{Alph}

\theoremstyle{plain}
\theorembodyfont{\upshape}
\theoremsymbol{\ensuremath{\blacklozenge}}

\newtheorem{definition}[lemma]{Definition}
\newtheorem{example}[lemma]{Example}

\crefname{definition}{definition}{definitions}
\crefformat{definition}{#2definition~#1#3} 
\Crefformat{definition}{#2Definition~#1#3} 

\crefname{ex}{example}{examples}
\crefformat{example}{#2example~#1#3} 
\Crefformat{example}{#2Example~#1#3} 

\crefname{exs}{example}{examples}
\crefformat{examples}{#2example~#1#3} 
\Crefformat{examples}{#2Example~#1#3} 

\crefname{remark}{remark}{remarks}
\crefformat{remark}{#2remark~#1#3} 
\Crefformat{remark}{#2Remark~#1#3} 

\crefname{remarks}{remark}{remarks}
\crefformat{remarks}{#2remark~#1#3} 
\Crefformat{remarks}{#2Remark~#1#3} 

\crefname{convention}{convention}{conventions}
\crefformat{convention}{#2convention~#1#3} 
\Crefformat{convention}{#2Convention~#1#3} 

\crefname{notation}{notation}{notations}
\crefformat{notation}{#2notation~#1#3} 
\Crefformat{notation}{#2Notation~#1#3} 

\crefname{table}{table}{tables}
\crefformat{table}{#2table~#1#3} 
\Crefformat{table}{#2Table~#1#3}

\crefname{lemma}{lemma}{lemmas}
\crefformat{lemma}{#2lemma~#1#3} 
\Crefformat{lemma}{#2Lemma~#1#3} 

\crefname{proposition}{proposition}{propositions}
\crefformat{proposition}{#2proposition~#1#3} 
\Crefformat{proposition}{#2Proposition~#1#3} 

\crefname{propositionN}{proposition}{propositions}
\crefformat{propositionN}{#2proposition~#1#3} 
\Crefformat{propositionN}{#2Proposition~#1#3} 

\crefname{corollary}{corollary}{corollaries}
\crefformat{corollary}{#2corollary~#1#3} 
\Crefformat{corollary}{#2Corollary~#1#3} 

\crefname{corollaryN}{corollary}{corollaries}
\crefformat{corollaryN}{#2corollary~#1#3} 
\Crefformat{corollaryN}{#2Corollary~#1#3} 

\crefname{theorem}{theorem}{theorems}
\crefformat{theorem}{#2theorem~#1#3} 
\Crefformat{theorem}{#2Theorem~#1#3} 

\crefname{theoremN}{theorem}{theorems}
\crefformat{theoremN}{#2theorem~#1#3} 
\Crefformat{theoremN}{#2Theorem~#1#3} 

\crefname{enumi}{}{}
\crefformat{enumi}{#2#1#3}
\Crefformat{enumi}{#2#1#3}

\crefname{assumption}{assumption}{Assumptions}
\crefformat{assumption}{#2assumption~#1#3} 
\Crefformat{assumption}{#2Assumption~#1#3} 

\crefname{construction}{construction}{Constructions}
\crefformat{construction}{#2construction~#1#3} 
\Crefformat{construction}{#2Construction~#1#3} 

\crefname{sketch}{sketch}{Sketches}
\crefformat{sketch}{#2sketch~#1#3} 
\Crefformat{sketch}{#2Sketch~#1#3} 

\crefname{recollection}{recollection}{Recollectiones}
\crefformat{recollection}{#2recollection~#1#3} 
\Crefformat{recollection}{#2Recollection~#1#3} 

\crefname{question}{question}{Questions}
\crefformat{question}{#2question~#1#3} 
\Crefformat{question}{#2Question~#1#3} 

\crefname{equation}{}{}
\crefformat{equation}{(#2#1#3)} 
\Crefformat{equation}{(#2#1#3)}

\numberwithin{equation}{section}
\renewcommand{\theequation}{\thesection-\arabic{equation}}

\theoremstyle{nonumberplain}
\theoremsymbol{\ensuremath{\blacksquare}}

\newtheorem{proof}{Proof}
\newcommand\pf[1]{\newtheorem{#1}{Proof of \Cref{#1}}}

\newcommand\bC{{\mathbb C}}

\newcommand\bG{{\mathbb G}}

\newcommand\bO{{\mathbb O}}

\newcommand\bU{{\mathbb U}}

\newcommand\bZ{{\mathbb Z}}

\newcommand\cC{{\mathcal C}}

\newcommand\cF{{\mathcal F}}
\newcommand\cG{{\mathcal G}}

\newcommand\cK{{\mathcal K}}
\newcommand\cL{{\mathcal L}}

\newcommand\cO{{\mathcal O}}
\newcommand\cP{{\mathcal P}}

\newcommand\cS{{\mathcal S}}

\newcommand\1{{\bf 1}}

\DeclareMathOperator{\Irr}{Irr}
\DeclareMathOperator{\id}{id}

\DeclareMathOperator{\im}{\mathrm{im}}

\DeclareMathOperator{\spn}{\mathrm{span}}
\DeclareMathOperator{\spec}{\mathrm{spec}}

\DeclareMathOperator{\supp}{\mathrm{supp}}

\DeclareMathOperator{\Tr}{\mathrm{Tr}}

\newcommand\numberthis{\addtocounter{equation}{1}\tag{\theequation}}

\newcommand{\cat}[1]{\textsc{#1}}

\newcommand{\qedhere}{\mbox{}\hfill\ensuremath{\blacksquare}}

\newcommand{\comment}[1]{}

\newcommand{\xrightarrowdbl}[2][]{%
  \xrightarrow[#1]{#2}\mathrel{\mkern-14mu}\rightarrow
}

\title{Spectral finiteness, quantum norm continuity and classical points}
\author{Alexandru Chirvasitu}

\begin{document}

\date{}

\newcommand{\Addresses}{{% additional braces for segregating \footnotesize
  \bigskip
  \footnotesize

  \textsc{Department of Mathematics, University at Buffalo}
  \par\nopagebreak
  \textsc{Buffalo, NY 14260-2900, USA}  
  \par\nopagebreak
  \textit{E-mail address}: \texttt{achirvas@buffalo.edu}

}}

\maketitle

\begin{abstract}
  We prove various notions of uniform continuity for compact-quantum-group representations on Hilbert or Banach spaces equivalent to having finite spectrum, i.e. finitely many isotypic components. This generalizes the classical analogue for compact-group representations on Banach spaces, and relies in part on Riemann-Lebesgue-type decay properties for Fourier coefficients of elements in minimal tensor products with compact-quantum-group function algebras. 
\end{abstract}


\vspace{.5cm}

\noindent{MSC 2020: 46L67; 22D10; 16T05; 16T15; 47B01; 43A32; 46L85; 47L65
  % 46L67 Quantum groups (operator algebraic aspects)
  % 22D10 Unitary representations of locally compact groups
  % 16T05 Hopf algebras and their applications
  % 16T15 Coalgebras and comodules; corings
  % 47B01 Operators on Banach spaces
  % 43A32 Other transforms and operators of Fourier type
  % 46L85 Noncommutative topology
  % 47L65 Crossed product algebras (analytic crossed products)
}

%\tableofcontents

%%%%%%%%%%%%%%%%%%%%%%%%%%%%%%%%
%%%%%%%%%%%%%%%%%%%%%%%%%%%%%%%%
\section*{Introduction}

The aim in the present note is to compare and relate two aspects of compact-quantum group unitary representations:
\begin{itemize}[wide]
\item on the one hand analytic: continuity/decay properties mimicking, noncommutative-geometrically, classical norm-continuity for a unitary representation $\bG\to \bU(H)$ of a compact group $\bG$ on a Hilbert space $H$;
\item on the other, representation-theoretic: the representation's having finite \emph{spectrum}, i.e. finitely many \emph{isotypic components}. 
\end{itemize}
Classically, the two properties (norm continuity and spectral finiteness) are indeed equivalent: \cite[p.257]{klm_unif} addresses this for connected second-countable groups, \cite[Corollary 2]{zbMATH05628052} in the broader setup of arbitrary compact groups represented norm-continuously on \emph{Banach} spaces, and \cite[Theorem 3.10]{Chirvasitu2026JNCG604} provides a number of equivalent characterizations.

The \emph{compact quantum groups} $\bG$ to which the above uniformity/spectral finiteness discussion is to be extended are here as in \cite[Definition 1.1.1]{NeTu13}: unital $C^*$-algebras $\cC(\bG)$ (thought of as continuous-function algebras on the otherwise-nonexistent $\bG$) equipped with a coassociative unital $C^*$ morphism
\begin{equation*}
  \cC(\bG)
  \xrightarrow{\quad\Delta\quad}
  \cC(\bG)\underline{\otimes}\cC(\bG)
  \quad
  \left(\text{spatial $C^*$ tensor product}\right)
\end{equation*}
with (the products' linear spans) $\left(\cC(\bG)\otimes 1\right)\Delta\cC(\bG)$ and $\left(1\otimes \cC(\bG)\right)\Delta\cC(\bG)$ both dense in $\cC(\bG)^{\underline{\otimes}2}$. 

We recall briefly in \Cref{se:qunif} below how unitary $\bG$-representations on Hilbert spaces $H$ are typically cast as unitary elements $U\in M(\cC(\bG)\underline{\otimes}\cK(H))$: multiplier algebra of the displayed minimal $C^*$ tensor product between the quantum function algebra $\cC(\bG)$ and the algebra of compact operators on $H$. That taken for granted, the quantization of the (uniformity=spectral finiteness) principle reads:

\begin{theorem}\label{th:fin.spec.equiv}
  The following conditions on a unitary $\bG$-representation $U\in M(\cC(\bG)\underline{\otimes}\cK(H))$ on a Hilbert space $H$ are equivalent.
  \begin{enumerate}[(a),wide]
  \item\label{item:th:fin.spec.equiv:fin.spc} $U$ has finite spectrum.
  \item\label{item:th:fin.spec.equiv:adj} The continuous adjoint action
    \begin{equation*}
      \cK(H)
      \xrightarrow{\quad U^*(1\otimes -)U\quad}
      \cC(\bG)\underline{\otimes}\cK(H)
    \end{equation*}
    extends to a continuous action on $\cL(H)$.

  \item\label{item:th:fin.spec.equiv:unif} $U$ is norm-continuous, in the sense that $U\in \cC(\bG)\underline{\otimes}\cL(H)$. 
  \end{enumerate}
\end{theorem}

There are alternative possible notions of norm continuity for compact-quantum-group representations, beyond those featuring in \Cref{th:fin.spec.equiv}. \cite[Definition 3.6(3)]{Chirvasitu2026JNCG604}, for instance, expands the scope of the discussion to representations on \emph{Banach} (as opposed to Hilbert) spaces $E$ in the sense of \cite[Definition 3.1]{Chirvasitu2026JNCG604}: coassociative maps $E\xrightarrow{\rho}\cC(\bG)\otimes_{\varepsilon} E$ into the \emph{injective} \cite[\S 3.1]{ryan_ban} Banach-space tensor product satisfying the customary density condition (see \Cref{eq:coact.e} below). In that Banach-space context, said norm (or \emph{uniform}) continuity is defined via \cite[Lemma 3.7]{Chirvasitu2026JNCG604} as
\begin{equation}\label{eq:ban.rpr.nrm}
  \cC(\bG)^*
  \ni
  \psi
  \xmapsto[\quad\text{weak$^*$-to-norm continuous}\quad]{\quad}
  \left(v\xmapsto{\quad}(\psi\otimes\id)\rho v\right)
  \in \cL(E).
\end{equation}
%We will also frequently refer to the property as \emph{uniformity}, for short.
Spectral machinery extends to Banach-space representations, and one direct quantum counterpart to the classical coincidence of uniformity with spectrum finiteness is

\begin{theorem}\label{th:ban.unif.cqg.cls.pt}
  For a compact quantum group $\bG$ represented via $E\xrightarrow{\rho}\cC(\bG)\otimes_{\varepsilon} E$ on a Banach space $E$ the following are equivalent.
  \begin{enumerate}[(a),wide]
  \item\label{item:th:ban.unif.cqg.cls.pt:fin.sp} $\spec \rho$ is finite.
  \item\label{item:th:ban.unif.cqg.cls.pt:unif} $\rho$ is norm-continuous in the sense that \Cref{eq:ban.rpr.nrm} holds.  \qedhere
  \end{enumerate}
\end{theorem}

This reverses the implication recorded as \cite[Lemma 3.8]{Chirvasitu2026JNCG604}, but note that the hypothesis is arguably somewhat excessive: classically, one need not require the full power of \Cref{eq:ban.rpr.nrm}, for in that case norm uniformity translates directly to the map in question being continuous when restricted to the space of \emph{pure states} of $\cC(\bG)$ (i.e. $\bG$ itself, with points identified with their respective Dirac-delta measures). This latter remark too has a quantum analogue, in the form of \Cref{th:if.lg.uaij} below. To make sense of the statement:
\begin{itemize}[wide]
\item we reference the matrix coefficients $u^{\alpha}_{ij}\in \cC(\bG)$ (see \Cref{se:qunif}'s brief recollection) with $\alpha$ ranging over irreducible $\bG$-representations $\Irr(\bG)$;
\item and will refer to a unitary $\bG$-representation $U$ as \emph{uniform}$_{\le 1}$ (\Cref{def:unif.bdd}) provided \Cref{eq:ban.rpr.nrm}'s map is assumed continuous only on the unit ball. 
\end{itemize}

\begin{theorem}\label{th:if.lg.uaij}
  Let $U\in M(\cC(\bG)\underline{\otimes}\cK(H))$ be a unitary representation of a compact quantum group $\bG$.
  \begin{enumerate}[(1),wide]
  \item\label{item:th:if.lg.uaij:temp.dec} If the \emph{Pontryagin dual} \cite[Theorem 3.3.2]{tim} $\Gamma:=\widehat{\bG}$ \emph{has tempered decay} in the sense that
    \begin{equation*}
      \begin{gathered}
        \exists\left(C>0\right)
        \forall\left(\alpha\in \Irr(\bG)\right)
        \left(\left\|u^{\alpha}\right\|_{\infty}^{\cC_r(\bG)}> C\right)\\
        \left\|u^{\alpha}\right\|_{\infty}^{\cC(\bG)}
        :=
        \sup_{\substack{v,w\in \bC^{\dim\alpha}\\\|v\|,\|w\|\le 1}}
        \left\|
          \Braket{v\mid u^{\alpha}\mid w}
          :=
          \sum_{i,j=1}^{\dim\alpha}u^{\alpha}_{ij}\overline{v_i}w_j
        \right\|,
      \end{gathered}      
    \end{equation*}
    then $U$ is uniform$_{\le 1}$ if and only if it has finite spectrum. 
  \item\label{item:th:if.lg.uaij:cls.pt} In particular, said equivalence holds if the reduced version of $\bG$ \emph{has a classical point}, in the sense that $\cC_r(\bG)$ has a multiplicative state. Equivalently: if $\bG$ is \emph{coamenable} \cite[Definition 3.1]{bt}. 
  \end{enumerate}
\end{theorem}

Counterexamples (see \Cref{th:rd.exp.dim}) show that one cannot \emph{completely} dispense with some type of decay control on the matrix coefficients of $\bG$. 

%%%%%%%%%%%%%%%%%%%%%%%%%%%%%%%%
\subsection*{Acknowledgments}

Initial drafts have benefited from M. Brannan's valuable input

% %

%%%%%%%%%%%%%%%%%%%%%%%%%%%%%%%%
%%%%%%%%%%%%%%%%%%%%%%%%%%%%%%%%
\section{Quantum uniform continuity in several guises, and consequences}\label{se:qunif}

Some of the usual operator-algebraic notation and conventions will be operative: $A\bullet B$ with $\bullet\in\{\otimes,\ \underline{\otimes},\ \overline{\otimes}\}$ stand for the ordinary algebraic, and \emph{spatial} (or \emph{minimal} or \emph{injective}) \cite[Definitions IV.4.8 and IV.5.1]{tak1} $C^*$ and $W^*$ tensor products respectively (between operator algebras of context-appropriate type), $M(-)$ is the \emph{multiplier algebra} \cite[Theorem II.7.3.1]{blk}, $\cL(-,-)$ and $\cK(-,-)$ denote spaces of bounded and compact operators respectively between Hilbert/Banach spaces, and we refer freely to \emph{slice maps} (\cite[\S\S II.9.7.1, III.2.2.6]{blk}, \cite[Lemma 1.5]{lprs}, \cite[\S 1]{MR567834}, etc.)
\begin{equation*}
  \begin{aligned}
    M\left(A\underline{\otimes} B\right)
    &\xrightarrow{\quad\psi\otimes \id\quad}
      M(B)\\
    A\overline{\otimes} B
    &\xrightarrow{\quad\psi\otimes \id\quad}
      B
  \end{aligned}
  ,\quad
  \psi\in A^*\left(:=\text{continuous dual of $A$}\right)
\end{equation*}

For compact quantum $\bG$ with underlying function algebra $\cC(\bG)$ we work throughout with Woronowicz's matrix coefficients (\cite[\S\S 1 and 4]{wor}, \cite[\S 1, pp.2-3]{podl_symm}, \cite[\S\S 1.3 and 1.4]{NeTu13}, etc.)
\begin{equation*}
  u^{\alpha}_{ij}\in \cO(\bG)
  ,\quad
  \begin{aligned}
    \alpha\in\Irr(\bG)
    &:=
      \left\{\text{iso classes of simple $\cO(\bG)$-comodules}\right\}\\
    1\le i,j\le d_{\alpha}
    &
      :=\dim\alpha
  \end{aligned}
\end{equation*}
spanning the (unique dense \cite[Theorem 3.1.7]{kt_qg-surv-1}) Hopf $*$-algebra of \emph{representative functions} on a compact quantum group $\bG$. \emph{Unitary $\bG$-representations} on Hilbert spaces $H$ are as in \cite[\S 2.2]{dsv}:
\begin{equation*}
  \text{unitary }U\in M(\cC(\bG)\underline{\otimes}\cK(H))
  ,\quad
  (\Delta\otimes \id)U = U_{13}U_{23},
\end{equation*}
with $U_{ij}$ denoting $U$ applied on the $i^{th}$ and $j^{th}$ tensorands of the argument. Just as classically \cite[Theorem 4.22]{hm5}, there is a rich Peter-Weyl theory of unitary-representation decompositions into irreducible summands isomorphic to the various $\alpha\in \Irr(\bG)$ and corresponding \emph{spectral projections}
\begin{equation*}
  \left(\psi^{\alpha}\otimes\id\right)U^*
  =:
  P^{\alpha}=\left(P^{\alpha}\right)^* = \left(P^{\alpha}\right)^2\in \cL(H)
\end{equation*}
with $\psi^{\alpha}\in \cC(\bG)^*$ the functionals of \cite[(4)]{podl_symm}, agreeing with the counit $\varepsilon$ of $\cO(\bG)$ on the span of $u^{\alpha}_{ij}$ and annihilating all $u_{k\ell}^{\alpha'\ne \alpha}$. The range $\im P^{\alpha}$ is the $\alpha$-isotypic component of $U$ (the largest subrepresentation isomorphic to a sum of copies of $\alpha$), and 
\begin{equation*}
  \spec U:=\left\{\alpha\in \Irr(\bG)\ :\ P^{\alpha}\ne 0\right\}
\end{equation*}
is the \emph{spectrum} of $U$.

It will be convenient to also assume the canonical positive operators
\begin{equation}\label{eq:falpha}
  F_{\alpha}\in \cL(V_{\alpha})
  ,\quad
  \Tr F_{\alpha}=\Tr F^{-1}_{\alpha}=:q_{\alpha}
  \quad
  \text{\cite[Theorem 5.4]{wor}}
\end{equation}
diagonal. One effect this will have is to render the $u^{\alpha}_{ij}$, (via \cite[(5.14) and (5.15)]{wor} or \cite[Theorem 1.4.3]{NeTu13}), mutually orthogonal with respect to both versions of the inner product induced by the \emph{Haar state} $h:=h_{\bG}$ of \cite[\S 4]{wor} or \cite[\S 1.2]{NeTu13}:
\begin{equation}\label{eq:inn.prds}
  \Braket{x\mid y}_r := h(x^*y)
  \quad\text{and}\quad
  \Braket{x\mid y}_{\ell} := h(x y^*),
\end{equation}
the index indicating the position (left/right) of the product's linear argument. $\cC_r(\bG)$ denotes the \emph{reduced} version of $\cG$, i.e. \cite[post Corollary 1.7.5]{NeTu13} the image of $\cC(\bG)$ through the GNS representation attached to $h$. 

For a unital $C^*$-algebra $A$ an element $x\in \cC(\bG)\underline{\otimes} A$ has a Fourier-series decomposition
\begin{equation*}
  x\sim\sum_{\substack{\alpha\in\Irr(\bG)\\1\le i,j\le d_{\alpha}}}u^{\alpha}_{ij}\otimes x^{\alpha}_{ij},
\end{equation*}
the ``$\sim$'' indicating heuristic rather than an equality in any formal sense (convergence issues being ignored); the rigorous version is rather
\begin{equation}\label{eq:xaij}
  x^{\alpha}_{ij}
  :=
  \frac{q_{\alpha}}{F^{-1}_{\alpha ii}}
  (h\otimes\id)\left(u^{\alpha *}_{ij}x\right)
  =
  \frac{q_{\alpha}}{F_{\alpha jj}}
  \left(h\otimes\id)(x u^{\alpha *}_{ij}\right),
\end{equation}
the $u^{\alpha}_{ij}$ being assumed $h$-orthogonal as previously announced. This suggests the equally informal
\begingroup
\allowdisplaybreaks
\begin{align*}
  xx^*
  &\sim
    \sum_{\substack{\alpha,i,j\\\beta,k,l}}u^{\alpha}_{ij}u^{\beta*}_{kl} x^{\alpha}_{ij}x^{\beta*}_{kl}
  \xRightarrow{\quad}\\
  E\left(xx^*\right)
  :=
  \left(h\otimes\id\right)(xx^*)  
  &=
    \sum_{\substack{\alpha,i,j\\\beta,k,l}}h\left(u^{\alpha}_{ij}u^{\beta*}_{kl}\right) x^{\alpha}_{ij}x^{\beta*}_{kl}\numberthis\label{eq:e.nrm.cvg}\\
  &=
    \sum_{\alpha,i,j}\frac 1{q_{\alpha}}F_{\alpha jj}x^{\alpha}_{ij}x^{\alpha *}_{ij}
    \quad\text{\cite[(5.14)]{wor}}.
\end{align*}
\endgroup

The extent to which this last equality holds formally, with the right-hand side norm-convergent, is the focus of amply studied Fourier decay-at-infinity phenomena: quantum analogues of the familiar \cite[\S I.2]{bc_four_1949} \emph{Riemann-Lebesgue lemma} in classical Fourier analysis, as alluded to briefly in \cite[\S 1, p.34]{MR3302101}, say. We revisit the matter in the course of the proof of \Cref{th:fin.spec.equiv}.  

Note, first, an immediate consequence of \Cref{th:fin.spec.equiv}:

\begin{corollary}\label{cor:rep.in.a}
For any compact quantum group $\bG$ and unital $C^*$-algebra $A$ every $\bG$-representation $U\in \cC(\bG)\underline{\otimes}A$ in $A$ in the sense of \cite[Proposition 5.2]{kus-univ} has finite spectrum.  \qedhere
\end{corollary}

\pf{th:fin.spec.equiv}
\begin{th:fin.spec.equiv}
  The equivalence \Cref{item:th:fin.spec.equiv:fin.spc} $\Leftrightarrow$ \Cref{item:th:fin.spec.equiv:adj} is covered by \cite[Theorem 3.16]{Chirvasitu2026JNCG604}, while \Cref{item:th:fin.spec.equiv:fin.spc} is plainly stronger than the other conditions: finite spectrum in fact entails
  \begin{equation*}
    U\in \cO(\bG)\otimes \cL(H)
    \quad
    \left(\text{algebraic tensor product}\right).
  \end{equation*}
  It remains to verify \Cref{item:th:fin.spec.equiv:unif} $\Rightarrow$ \Cref{item:th:fin.spec.equiv:fin.spc}, which task the proof henceforth turns to. Reprising the notation of \Cref{eq:e.nrm.cvg} for the usual \emph{conditional expectation} \cite[Definition II.6.10.1]{blk} $\cC(\bG)\underline{\otimes}\cL(H)\xrightarrowdbl{E}\cL(H)$, \Cref{pr:str2nrm} below ensures that that equation in fact holds as written, with a \emph{norm}-convergent right-hand side. Applied to $x:=U$, \Cref{eq:e.nrm.cvg} reads
  \begingroup
  \allowdisplaybreaks
  \begin{align*}
    1
    =
    E(1)
    =
    E\left(xx^*\right)
    &=
      \sum_{\alpha,j}\frac 1{q_{\alpha}}F_{\alpha jj}\sum_i x^{\alpha}_{ij}x^{\alpha *}_{ij}\\
    &=
      \sum_{\alpha}\left(\sum_j\frac 1{q_{\alpha}}F_{\alpha jj}\right) P^{\alpha}\\
    &=
      \sum_{\alpha} P^{\alpha}
  \end{align*}
  \endgroup
  as \emph{norm} convergence, where
  \begin{equation}\label{eq:spc.idemp.u}
    P^{\alpha}:=(\psi^{\alpha}\otimes \id)U^*\in \cL(H)
    ,\quad
    \cC(\bG)
    \ni
    u^{\beta}_{ij}
    \xmapsto{\quad\psi^{\alpha}\quad}
    \delta_{\alpha\beta}\delta_{ii}
    \in \bC
  \end{equation}
  are the respective projections (introduced in \cite[Theorem 1.5]{podl_symm} as $E^{\alpha}$) onto the \emph{$\alpha$-isotypic components} (cf. \cite[remarks preceding Proposition 13]{boc_zb}) of the representation $U$. Said norm convergence of mutually orthogonal projections, naturally, requires that only finitely many be non-zero. 
\end{th:fin.spec.equiv}

The proof of the auxiliary \Cref{pr:str2nrm} below will precisely parallel its analogue \cite[Lemma 5.2]{MR2873171} in the context of \emph{reduced crossed products} \cite[Definition 4.1.4]{bo_cast-approx_2008} resulting from actions by countable (classical, discrete) groups on $C^*$-algebras (so in our language, that result's $\bG$ would be an \emph{abelian}, i.e. dual-classical, compact quantum group). That source leverages weaker convergence into its stronger, norm counterpart via a device credited there to G. Elliott and U. Haagerup and reliant on \emph{Dini's} \cite[Theorem 12.1]{jst_post-an_3e_2005}\footnote{Said invocation of Dini, typically stated for sequences of real-valued functions on compact Hausdorff spaces, appears to be the one aspect of the proof of \cite[Lemma 5.2]{MR2873171} relying in any substantive fashion on any countability hypotheses. We will see in the course of the proof of \Cref{pr:str2nrm} that such countability assumptions can be dispensed with.} (see also \cite[Lemma 3.1]{mnw} for essentially the same gadget).

\begin{proposition}\label{pr:str2nrm}
  For a compact quantum group $\bG$, unital $C^*$-algebra $A$ and $x\in \cC(\bG)\underline{\otimes}A$ we have the norm-convergent expansions
  \begin{equation*}
    E(xx^*)
    =
    \sum_{\alpha,i,j}\frac 1{q_{\alpha}}F_{\alpha jj}x^{\alpha}_{ij}x^{\alpha *}_{ij}
    \quad\text{and}\quad
    E(x^*x)
    =
    \sum_{\alpha,k,\ell}\frac 1{q_{\alpha}}F_{\alpha kk}x^{\alpha *}_{k\ell}x^{\alpha}_{k\ell}
  \end{equation*}
  for the respective values of the expectation $\cC(\bG)\underline{\otimes}A\xrightarrow{E:=h\otimes\id}A$.
\end{proposition}
\begin{proof}
  With nothing of any substance distinguishing the two branches, we address the first claim only. Write
  \begin{itemize}[wide]
  \item $A\xrightarrow{\pi}\cL(H)$ for a representation of $A$;
  \item $\cC(\bG)\xrightarrow{\pi_{h}}\cL\left(L^2(\bG)\right)$ for the GNS representation of the Haar state $h\in \cC(\bG)^*$: the completion of $\cO(\bG)$ under the inner product $\Braket{-\mid -}_r$ of \Cref{eq:inn.prds};
  \item and
    \begin{equation*}
      H\ni \xi
      \xmapsto[\quad\substack{\alpha\in \Irr(\bG)\\1\le i,j\le d_{\alpha}}\quad]{\quad T^{\alpha}_{ij}\quad}
      u^{\alpha*}_{ij}\otimes \xi
      \in L^2(\bG)\otimes H.
    \end{equation*}
  \end{itemize}
  Observe, next, the convergence
  \begin{equation*}
    \sum_{\alpha,i,j}
    \frac{q_{\alpha}}{F_{\alpha jj}}
    T^{\alpha}_{ij}T^{\alpha *}_{ij}
    =1
    \in
    \cL\left(L^2(\bG)\otimes H\right)
  \end{equation*}
  in the \emph{$\sigma$-strong$^*$} topology of \cite[Definition II.2.3]{tak1} (the strongest of the locally convex sub-norm topologies examined in \cite[\S II.2]{tak1}): the individual summands $\displaystyle\frac{q_{\alpha}}{F_{\alpha jj}}T^{\alpha}_{ij}T^{\alpha *}_{ij}$ are (by \cite[(5.14)]{wor} again) mutually orthogonal projections whose ranges exhaust the underlying Hilbert space. 

  Note also:
  \begin{equation*}
    \forall\left(\alpha\in \Irr(\bG)\right)
    \forall\left(1\le i,j\le d_{\alpha}\right)
    \bigg(
    \pi x^{\alpha}_{ij}
    =
    \frac{q_{\alpha}}{F_{\alpha jj}}
    T^{\mathbf{1}*}_{11}\left(\pi_h\otimes \pi\right)(x)T^{\alpha}_{ij}
    \bigg)
  \end{equation*}
  ($\mathbf{1}\in \Irr(\bG)$ denoting the trivial representation). As in the aforementioned \cite[Lemma 5.2]{MR2873171}, this now entails the pointwise convergence
  \begin{equation*}
    \pi E(xx^*)
    =
    \sum_{\alpha,i,j}\frac {F_{\alpha jj}}{q_{\alpha}}\pi \left(x^{\alpha}_{ij}x^{\alpha *}_{ij}\right)
  \end{equation*}
  on the space of states on $A$ inherited from the weak$^*$ ones on $\cL(H)$ via $\pi$. We can arrange for this to be the entire state space $\cS(A)$ (take for $\pi$ the \emph{universal representation} \cite[\S 3.7.6]{ped-aut}), hence also \emph{uniform} convergence by the version of Dini's theorem proven in \cite{MR1764529}: valid for monotone \emph{nets} of functions valued in arbitrary uniform spaces (monotonicity being defined appropriately to that context) rather than real-valued-function \emph{sequences}.
\end{proof}

Recall next the notion of $\bG$-representation on a Banach space $E$ for a compact quantum group $\bG$: 
\begin{equation}\label{eq:coact.e}
  \begin{gathered}
    E
    \xrightarrow[\quad\text{coassociative}\quad]{\quad\rho\quad}
    \cC(\bG)\otimes_{\varepsilon} E
    \quad
    \left(\text{\emph{injective} \cite[\S 3.1]{ryan_ban} Banach tensor product}\right)    \\
    \overline{
      \spn
      \left\{
        (x\otimes \id)\rho v\ :\ x\in \cC(\bG),\ v\in E
      \right\}
    }
    =
    \cC(\bG)\otimes_{\varepsilon} E.
  \end{gathered}  
\end{equation}
One again has spectral idempotents $P^{\alpha}:=(\psi^{\alpha}\otimes\id)\circ\rho\in \cL(E)$ analogous to \Cref{eq:spc.idemp.u} and, by \cite[Proposition]{Chirvasitu2026JNCG604}, the more familiar unitary representations $U\in M(\cC(\bG)\underline{\otimes}\cK(H))$ can be recast as representations on the underlying Banach spaces by
\begin{equation}\label{eq:uast.1v}
  H\ni v
  \xmapsto{\quad}
  U^*(1\otimes v)
  \in
  \cC(\bG)\otimes_{\varepsilon} H.
\end{equation}

% %

\pf{th:ban.unif.cqg.cls.pt}
\begin{th:ban.unif.cqg.cls.pt}
  While \cite[Lemma 3.8]{Chirvasitu2026JNCG604} notes only that having finite spectrum implies \Cref{eq:ban.rpr.nrm}, the converse is also straightforward: weak$^*$-to-norm continuity is an extremely strong condition to impose, holding only in trivial circumstances.

  To elaborate, consider a linear map $E\xrightarrow{T}F$ between \emph{locally convex} \cite[\S 18.1]{k_tvs-1} topological vector spaces. Its continuity with respect to
  \begin{itemize}[wide]
  \item the \emph{weak (or simple) topology} ($\mathfrak{T}_s$ in \cite[\S 20.2]{k_tvs-1});
  \item and the given topology on $F$, assuming some origin neighborhood $U$ contains no non-trivial linear subspaces of $F$ (so in particular for normable $F$)
  \end{itemize}
  automatically entails the finite dimensionality of $\im T$. This is little more than an unpacking of the hypotheses: there must be finitely many continuous functionals $f_i\in E^*$ with
  \begin{equation*}
    \forall\left(v\in \bigcap_i \ker f_i\right)
    \left(Tv\in U\right).
  \end{equation*}
  One can scale $v$ at will, so the assumption on $U$ forces $T$ to vanish on the finite-codimensional $\bigcap_i \ker f_i$. 
\end{th:ban.unif.cqg.cls.pt}

In view of the preceding proof, rendering \Cref{th:ban.unif.cqg.cls.pt} immediate by means of unrestrained weak continuity, one might attempt adjacent refinements to uniformity.

\begin{definition}\label{def:unif.bdd}
  A Banach-space representation \Cref{eq:coact.e} of a compact quantum group is \emph{boundedly uniform} or \emph{uniform$_{\le 1}$} if the continuity in \Cref{eq:ban.rpr.nrm} holds on norm-bounded subsets of $\cC(\bG)^*$ (equivalently, on any one norm-bounded origin neighborhood, e.g. the unit ball). 
\end{definition}

For classical $\bG$ bounded and plain uniformity are equivalent: as noted in the Introduction following \Cref{th:ban.unif.cqg.cls.pt}, uniformity$_{\le 1}$ already implies spectrum finiteness and hence \cite[Lemma 3.8]{Chirvasitu2026JNCG604} also ordinary uniformity. It is not inapposite, then, to ask whether the equivalence survives quantization.

As further motivation and surrounding context for \Cref{th:if.lg.uaij}, recall \cite[Remark 1.3]{2602.18878v1} to the effect that the classical \cite[Corollary 2]{zbMATH05628052} characterization of norm-continuous compact-group representations as precisely those with finite spectra is recovered by non-classical means from \cite[Theorem 0.1]{2602.18878v1} for compact quantum groups $\bG$ with a \emph{classical point} (i.e. with $\cC(\bG)$ admitting a multiplicative state); by extension, the same goes for \Cref{th:fin.spec.equiv} (which removes the classical-point constraint). \Cref{th:if.lg.uaij} describes how, upon assuming a small amount of classical behavior, said classical result replicates in terms of what is perhaps the most straightforward analogue of quantum uniform continuity for a unitary $\bG$-representation.

An initial remark will help set the stage for further investigation. 

\begin{lemma}\label{le:u.lleg}
  A unitary representation $U\in M(\cC(\bG)\underline{\otimes}\cK(H))$ of a compact quantum group $\bG$ is uniform$_{\le 1}$ as a Banach-space representation if and only if either one of 
  \begin{equation*}
    \cC(\bG)^*\ni \psi
    \xmapsto{\quad}
    (\psi\otimes\id)U
    \text{ or }
    (\psi\otimes\id)U^*
    \in
    \cL(H)
  \end{equation*}
  is weak$^*$-to-norm continuous on the unit ball of $\cC(\bG)^*$.
\end{lemma}
\begin{proof}
  For the latter map this is immediate given the conversion procedure \Cref{eq:uast.1v} from unitary to Banach-space representations, while the passage between the two is effected by $\left(\left(\psi\otimes\id\right)U\right)^*=\left(\psi^*\otimes\id\right)U^*$ for $\psi^*:=\overline{\psi\left(\bullet^*\right)}$.
\end{proof}

The device isolated by \Cref{pr:cast2nm.sym.legs} below, implicitly also driving the proof of \cite[Proposition 2.2]{MR567834}, will allow some symmetrization in the condition on $U$ imposed by \Cref{le:u.lleg}. Some vocabulary will help compress the statement.

\begin{definition}\label{def:lr.unif}
  For von Neumann algebras $M$ and $N$ and a property $\cP$ that a map $M^*\to N$ may or may not have, $x\in M\overline{\otimes} N$ is said to have property $\cP_{\ell}$ or to \emph{be} $\cP_{\ell}$ (for ``left'') if the slice map
  \begin{equation*}    
    M^*\ni
    \varphi
    \xmapsto{\quad}
    \left(\varphi\otimes \id\right)x
    \in N
  \end{equation*}
  induced by $x$ is $\cP$.

  The right-handed mirror is phrased in the expected fashion: $\cP_r$ in place of $\cP_{\ell}$, for the package involving $\id\otimes \varphi$ in place of $\varphi\otimes \id$. The terminology further applies in its guessable $C^*$ incarnation: $x\in M\underline{\otimes}N$.  
\end{definition}

One relevant property $\cP$ is the uniformity of \Cref{eq:ban.rpr.nrm} and (of more interest here) its bounded variant of \Cref{def:unif.bdd}. 

\begin{proposition}\label{pr:cast2nm.sym.legs}
  For $x\in M\overline{\otimes}N$ in the $W^*$ tensor product of two von Neumann algebras the following conditions are equivalent.
  \begin{enumerate}[(a),wide]
  \item\label{item:pr:cast2nm.sym.legs:unif.l} $x$ is uniform$_{\le 1,\ell}$.

  \item\label{item:pr:cast2nm.sym.legs:unif.0.l} $x$ is uniform$_{\le 1,\ell}$ with respect to (bounded) nets of weak$^*$ functionals converging to $0$.
    
  \item\label{item:pr:cast2nm.sym.legs:unif.w0.l} As in \Cref{item:pr:cast2nm.sym.legs:unif.0.l}, for nets of weak functionals instead.
    
  \item\label{item:pr:cast2nm.sym.legs:unif.vect0.l} Setting $\omega_{v,w}:=\Braket{v\mid \bullet w}\in M_*$ for $v,w\in H$ in a Hilbert space on which $M$ is represented faithfully as $M\le \cL(H)$,
    \begin{equation*}
      \omega_{v,w}
      \xmapsto{\quad}
      (\omega_{v,w}\otimes \id)x
      \in
      \cC(\bG)
    \end{equation*}
    turns weak$^*$-Cauchy nets  into their norm counterparts.
    
  \item\label{item:pr:cast2nm.sym.legs:unif.all.r} All or any of the above, in the right-handed version. 
  \end{enumerate}
\end{proposition}
\begin{proof}
  The implications
  \begin{equation*}
    \text{\Cref{item:pr:cast2nm.sym.legs:unif.l}}_{\bullet}
    \xRightarrow{\quad}
    \text{\Cref{item:pr:cast2nm.sym.legs:unif.0.l}}_{\bullet}
    \xRightarrow{\quad}
    \text{\Cref{item:pr:cast2nm.sym.legs:unif.w0.l}}_{\bullet}
    ,\
    \text{\Cref{item:pr:cast2nm.sym.legs:unif.l}}_{\bullet}
    \xRightarrow{\quad}
    \text{\Cref{item:pr:cast2nm.sym.legs:unif.vect0.l}}_{\bullet}
    ,\quad
    \bullet\in\left\{\ell,r\right\}
  \end{equation*}
  are all formal.

  Note that \Cref{item:pr:cast2nm.sym.legs:unif.l}$_{\ell}$ concerns all of $M^*$, whereas the other $\bullet_{\ell}$ specialize to (at least) weak$^*$-continuous functionals. The proof of \cite[Proposition 2.2]{MR567834}, in present language, shows effectively that the strongest left-handed statement \Cref{item:pr:cast2nm.sym.legs:unif.l}$_{\ell}$ is implied by its right-handed counterpart phrased for nets $(\psi_{\lambda})_{\lambda}\subset N_{*,\le 1}$ (unit ball of the predual of $N$) assumed convergent $N^*_{\le 1}$. This is easily recast as \Cref{item:pr:cast2nm.sym.legs:unif.0.l}$_r$ though: the net in question is certainly Cauchy and hence $(\psi_{\lambda}-\psi_{\lambda'})_{\lambda,\lambda'}$ is 0-convergent; consequently,
  \begin{equation*}
    \text{\Cref{item:pr:cast2nm.sym.legs:unif.0.l}$_r$}
    \xRightarrow{\quad}
    \left(
      \left(\id\otimes\psi_{\lambda}\right)x
    \right)_{\lambda}
    \text{ norm-convergent}
    \quad
    \left(\text{by norm completeness}\right).
  \end{equation*}
  \Cref{item:pr:cast2nm.sym.legs:unif.w0.l}$_{\bullet}$ $\Rightarrow$ \Cref{item:pr:cast2nm.sym.legs:unif.0.l}$_{\bullet}$ by weak-in-weak$^*$ norm density \cite[Theorem II.2.6(iii)]{tak1}.
  
  In order to complete the proof of the proposition as a whole by confirming that \Cref{item:pr:cast2nm.sym.legs:unif.vect0.l}$_{r}$ implies \Cref{item:pr:cast2nm.sym.legs:unif.l}$_{\ell}$, it remains to observe that the aforementioned net $(\psi_{\lambda})_{\lambda}\subset N_{*,\le 1}$ can be chosen so as to range over any set $\cS\subset N_{\le 1}^*$ detecting norms in $N$ in the following sense:
  \begin{equation*}
    \forall\left(T\in N\right)
    \forall\left(\varepsilon>0\right)
    \exists\left(\psi\in \cS\right)
    \left(|\psi T|\ge \|T\|-\varepsilon\right). 
  \end{equation*}
  Plainly, the set $\cS$ of all functionals $\omega_{v,w}$ with norm-$(\le 1)$ $v,w\in H$ will do.
\end{proof}

The aforementioned symmetrization afforded by \Cref{pr:cast2nm.sym.legs} refers to switching \Cref{le:u.lleg}'s criterion to $U$'s right-hand leg, as will become apparent presently.

\pf{th:if.lg.uaij}
\begin{th:if.lg.uaij}
  \begin{enumerate}[label={},wide]
  \item\textbf{\Cref{item:th:if.lg.uaij:temp.dec}} Assume infinitely many spectral projections $P^{\alpha}$ of $U$ non-zero, entailing the corresponding non-vanishing Fourier coefficients $x^{\alpha}_{ij}$, $1\le i,j\le \dim\alpha$ of $x:=U$ defined by \Cref{eq:xaij} (having fixed the $u^{\alpha}$-engendering bases of the $\alpha$ as explained preceding \Cref{eq:inn.prds}, so as to ensure the $u^{\alpha}_{ij}$ are mutually orthogonal). The tempered-decay hypothesis ensures the existence of norm-1
    \begin{equation*}
      v_{\alpha 0}=\left(v_{\alpha 0i}\right)_{1\le i\le \dim\alpha}
      \quad\text{and}\quad
      w_{\alpha 0}=\left(w_{\alpha 0i}\right)_{1\le i\le \dim\alpha}
      \quad:\quad
      \left\|
        \sum_{i,j=1}^{\dim\alpha}u^{\alpha}_{ij}\overline{v_{\alpha 0i}}w_{\alpha 0j}
      \right\|
      >C.
    \end{equation*}
    Now consider $\omega_{\alpha}:=\omega_{v_{\alpha},w_{\alpha}}:=\Braket{v_{\alpha}\mid \bullet w_{\alpha}}\in \cL(H)_*$ for
    \begin{equation*}
      v_{\alpha}:=\sum _{i=1}^{\dim\alpha} v_{\alpha 0i}e_i
      \quad\text{and}\quad
      w_{\alpha}:=\sum _{i=1}^{\dim\alpha} w_{\alpha 0i}e_i
      ,\quad
      e_i\in \im x^{\alpha}_{ii}\le \im P^{\alpha}\text{ orthonormal},
    \end{equation*}
    so that by construction we have
    \begin{equation*}
      \left\|(\id\otimes\omega_{\alpha})U\right\|
      =
      \left\|
        \sum_{i,j=1}^{\dim\alpha}u^{\alpha}_{ij}\overline{v_{\alpha 0i}}w_{\alpha 0j}
      \right\|
      >C.
    \end{equation*}
    By \emph{Alaoglu's} \cite[Theorem V.3.1]{conw_fa}, the bounded family $\{\omega_{\alpha}\}_{\alpha}$ supports a weak$^*$-Cauchy net which we index by the $\alpha$ themselves, with the understanding that $\alpha$ eventually avoids every finite subset of $\Irr(\bG)$. 

    Now, on the one hand the net $\left((\id\otimes\omega_{\alpha})U\right)_{\alpha}$ must norm-converge in $\cC_r(\bG)$ to an element of norm $\ge C>0$. On the other, matrix-coefficient Haar orthogonality forces
    \begin{equation*}
      \forall\left(\alpha\ne \alpha'\in \Irr(\bG)\right)
      \left(
        \Braket{(\id\otimes\omega_{\alpha})U\mid (\id\otimes\omega_{\alpha'})U}_r=0
      \right)
      \quad
      \left(\text{in \Cref{eq:inn.prds}'s notation}\right).
    \end{equation*}
    $h$ being faithful \cite[Corollary 1.7.5]{NeTu13} on $\cC_r(\bG)$, the desired conclusion that uniformity$_{\le 1}$ must fail follows from that property's characterization as \Cref{pr:cast2nm.sym.legs}'s right-handed \Cref{item:pr:cast2nm.sym.legs:unif.vect0.l}.
    
  \item\textbf{\Cref{item:th:if.lg.uaij:cls.pt}} $\bG$ has a classical point precisely when the counit $\varepsilon$ has a continuous extension to all of $\cC(\bG)$ (see \cite[paragraph preceding Theorem 2.2]{zbMATH07377294}; also \Cref{le:clsc.pt.iff} below for additional equivalent characterizations), whence a state on $\cC_r(\bG)$ assigning the value $1$ to every $u^{\alpha}_{ii}$. Those elements, then, must all have norm 1 and the conclusion follows from part \Cref{item:th:if.lg.uaij:temp.dec}.  \qedhere
  \end{enumerate}
\end{th:if.lg.uaij}

In reference to the property featuring in \cite[Theorem 0.1]{2602.18878v1} of compact quantum groups having at least one classical point (one version of the above-mentioned ``partial classical behavior''), \Cref{le:clsc.pt.iff} records a simple remark for whatever future-reference purpose it may serve. We will be referring to the usual \emph{convolution product} \cite[Definition 1.4.1]{mont}, denoted here simply by ``$\cdot$'' or juxtaposition, and defined as the precomposition
\begin{equation*}
  \cO(\bG)^{*\otimes 2}
  \xrightarrow{\quad \circ\Delta\quad}
  \cO(\bG)^*
\end{equation*}
on the algebraic dual of a coalgebra. Note also that it (or rather its appropriate extension) makes the state space $\cS(\cC(\bG))$ of the $C^*$-algebra $\cC(\bG)$ into a compact Hausdorff \cite[Theorem V.3.1]{conw_fa} semigroup under its weak$^*$ topology. The same goes for the unit ball $\cC(\bG)^*_{\le 1}$ in the continuous dual of $\cC(\bG)$: the norm dual to $\underline{\otimes}$ being \emph{cross} \cite[Proposition T.3.11]{wo},
\begin{equation*}
  \|\varphi\|,\ \|\psi\| \le 1
  \xRightarrow{\quad}
  \|\varphi\cdot\psi:=(\varphi\otimes \psi)\circ\Delta\|\le 1.
\end{equation*}
The compact semigroups
\begin{equation*}
  \cS:=\cS_{\bG}:=\cS(\cC(\bG))
  \quad\text{and}\quad
  \cK:=\cK_{\bG}:=\cC(\bG)^*_{\le 1}
\end{equation*}
will feature again below.

\begin{lemma}\label{le:clsc.pt.iff}
  For a compact quantum group with underlying function algebra $\cC(\bG)$ the following conditions are equivalent.
  \begin{enumerate}[(a)]
  \item\label{item:le:clsc.pt.iff:cls.pt} $\bG$ has a classical point, in the sense that $\cC(\bG)$ has a multiplicative state.
  \item\label{item:le:clsc.pt.iff:counit} The counit of the Hopf $*$-algebra $\cO(\bG)$ extends continuously to $\cC(\bG)$.
  \item\label{item:le:clsc.pt.iff:conv.inv} One of the semigroups $\cS_{\bG}$, $\cK_{\bG}$ or $\cC(\bG)^*$ has an invertible element.
  \item\label{item:le:clsc.pt.iff:bij} One of the semigroups $\cS_{\bG}$ or $\cK_{\bG}$ has an element $\psi$ with either multiplication map $\psi\cdot $ or $\cdot\psi$ bijective.
  \item\label{item:le:clsc.pt.iff:surj} One of the semigroups $\cS_{\bG}$ or $\cK_{\bG}$ has an element $\psi$ with either multiplication map $\psi\cdot $ or $\cdot\psi$ surjective.
  \end{enumerate}
\end{lemma}
\begin{proof}
  We have
  \begin{equation*}
    \text{\Cref{item:le:clsc.pt.iff:cls.pt}}
    \xLeftrightarrow{\text{\cite[pre Theorem 2.2]{zbMATH07377294}}}
    \text{\Cref{item:le:clsc.pt.iff:counit}}
    \xRightarrow{\text{plainly}}
    \text{\Cref{item:le:clsc.pt.iff:conv.inv}}_\bullet
    ,\quad
    \bullet\in\left\{\cS_{\bG},\ \cK_{\bG},\ \cC(\bG)^*\right\}
  \end{equation*}
  and
  \begin{equation*}
    \text{\Cref{item:le:clsc.pt.iff:conv.inv}}_\bullet
    \xRightarrow{\quad}
    \text{\Cref{item:le:clsc.pt.iff:bij}}_\bullet
    \xRightarrow{\quad}
    \text{\Cref{item:le:clsc.pt.iff:surj}}_\bullet
    ,\quad
    \bullet\in \left\{\cS_{\bG},\ \cK_{\bG}\right\},
  \end{equation*}
  also with $\text{\Cref{item:le:clsc.pt.iff:conv.inv}}_\bullet$ progressively weaker as the decorating semigroup $\bullet$ increases. It thus remains to argue that either one of \Cref{item:le:clsc.pt.iff:conv.inv}$_{\cC(\bG)^*}$ or \Cref{item:le:clsc.pt.iff:surj}$_{\bullet}$ implies \Cref{item:le:clsc.pt.iff:counit}. 

  \begin{enumerate}[label={},wide]
  \item\textbf{Claim: the unit of any of the semigroups $\cS$, $\cK$ or $\cC(\bG)^*$, if it exists, must be the counit.} That the variants of the claim are mutually equivalent follows from the fact that arbitrary continuous functionals are expressible as linear combinations of at most four states (e.g. by \cite[Theorem III.4.2(ii)]{tak1}). Having thus reduced the problem to its $\cC(\bG)^*$ variant, write $\psi$ for the assumed unit of that convolution monoid and
    \begin{equation*}
      \varepsilon_{\cF}:=
      \begin{cases}
        \varepsilon&\text{ on }\spn\left\{u^{\alpha}_{ij}\ :\ \alpha\in \cF\right\}\\
        0&\text{ on }\spn\left\{u^{\beta}_{ij}\ :\ \beta\not\in \cF\right\}
      \end{cases}
      ,\quad
      \text{finite }\cF\subseteq \Irr(\bG). 
    \end{equation*}
    The $\varepsilon_{\cF}$ being restrictions to $\cO(\bG)$ of continuous functionals (i.e. members of $\cC(\bG)^*$) by \cite[(5.37) and surrounding discussion]{wor},
    \begin{equation*}
      \psi
      =
      \lim_{\cF\nearrow}\psi\cdot \varepsilon_{\cF}
      =
      \lim_{\cF\nearrow}\varepsilon_{\cF}
      =
      \varepsilon
    \end{equation*}
    in the weak$^*$ topology on the full linear dual $\cO(\bG)^*$.

  \item\textbf{\Cref{item:le:clsc.pt.iff:conv.inv}$_{\cC(\bG)^*}$  or \Cref{item:le:clsc.pt.iff:surj}$_{\bullet}$ $\Rightarrow$ \Cref{item:le:clsc.pt.iff:counit}:} The \Cref{item:le:clsc.pt.iff:conv.inv} branch is immediate given the claim, so we focus on \Cref{item:le:clsc.pt.iff:surj}$_{\cK}$ to fix ideas. Being valued in a compact Hausdorff space, the sequence $\left(\psi^{n}\right)$ will have \cite[Theorem 17.4]{wil_top} a convergent (and hence also \emph{Cauchy} \cite[Definition 39.1]{wil_top}) \emph{subnet} $\left(\psi^{n_{\lambda}}\right)_{\lambda}$. As $\psi^{{n_{\lambda,\lambda'}}}$ are arbitrarily close for large $\lambda,\lambda'$, so will the respective convolution-multiplication maps
    \begin{equation*}
      \psi^{{n_{\lambda,\lambda'}}}\cdot
      \in
      \left(\cat{Cont}\left(\cK\to \cK\right),\ \text{uniform topology}\right)
    \end{equation*}
    for $\cK$: this follows from the identification \cite[Proposition 7.1.5]{brcx_hndbk-2}
    \begin{equation*}
      \cat{Cont}\left(\cK\times \cK\to \cK\right)
      \cong
      \cat{Cont}\left(\cK\to \cat{Cont}\left(\cK\to \cK\right)\right).
    \end{equation*}
    Consequently (``$\sim$'' meaning ``close''):
    \begin{equation*}
      \psi^{n_{\lambda'}}\cdot
      \sim
      \psi^{n_{\lambda}}\cdot
      \text{ on }
      \cK
      \xRightarrow{\quad}
      \psi^{\left(n_{\lambda'}-n_{\lambda}\right)}\cdot
      \sim
      \id
      \text{ on }
      \psi^{n_{\lambda}}
      \cK
      \xlongequal[\quad]{\text{surjectivity}}
      \cK.
    \end{equation*}
    The net $\left(\psi^{\left(n_{\lambda'}-n_{\lambda}\right)}\right)_{\lambda'>\lambda}$ will thus have a subnet converging to the identity, which must be a counit for $\cC(\bG)$ by the claim preceding the current portion of the argument.    
  \end{enumerate}
\end{proof}

\Cref{th:if.lg.uaij} leaves unaddressed the now-inevitable question of whether the decay condition can be removed entirely. Due to (purely quantum) peculiarities, it turns out it cannot. To frame the result formally as \Cref{th:rd.exp.dim} below, we need to recollect some language attendant to quantum-group growth/decay considerations.
\begin{itemize}[wide]
\item A \emph{finitely generated (f.g.)} discrete quantum group is the Pontryagin dual $\Gamma:=\widehat{\bG}$ for compact quantum $\bG$ whose underlying Hopf $*$-algebra $\cO(\bG)$ is f.g. as an algebra.
  
\item The \emph{word length function} \cite[Example 3.2(ii)]{MR2329000} $\ell$ on $\Gamma$ above with respect to a finite generating set $\cF\subseteq \Irr(\bG)$ (typically assumed symmetric, i.e. with the dual $\alpha^*\in \cF$ whenever $\alpha\in \cF$) is
  \begin{equation*}
    \Irr(\bG)
    \ni \alpha
    \xmapsto{\quad\ell=\ell_{\cF}\quad}
    \min\left\{k\in \bZ_{\ge 0}\ :\ \exists\left(\alpha_i\in \cF,\ 1\le i\le k\right)\left(\alpha\le \alpha_1\otimes\cdots\otimes\alpha_k\right)\right\}. 
  \end{equation*}
  We extend the notation $\ell(-)$ to elements of $\cO(\bG)$: one such has length $k$ by definition if it is a linear combination of matrix coefficients $u^{\alpha}_{ij}$ with $\ell(\alpha)=k$. 
\item $\Gamma$ has \emph{the property of rapid decay (RD)} \cite[Proposition and Definition 3.5]{MR2329000} (with respect to $\ell_{\cF}$) if
  \begin{equation*}
    \left(\forall x\in \cO(\bG)\right)
    \left(\|x\|\le P(\ell(x))\|x\|_2\right)
    ,\quad
    \|-\|_2:=h\left(-^*\cdot -\right)^{1/2}
  \end{equation*}
  for some polynomial $P$.

\item The terminology applies here (somewhat unconventionally) also ``locally'' or relatively, to families of representations: given any symmetric finite $\cF\subseteq \Irr(\bG)$ one might consider sequences $\alpha_n$ exhibiting rapid decay with respect to $\ell:=\ell_{\cF}$ in a sense the reader will have no difficulties extrapolating from the global version just sketched.
\end{itemize}

\Cref{th:rd.exp.dim} shows (e.g. via \Cref{ex:oplus.uplus}) that there do exist uniform$_{\le 1}$ representations with infinite spectrum; an auxiliary remark precedes it.

For $v\in H$ we write
\begin{equation}\label{eq:suppv}
  \supp v:=\left\{\alpha\in \Irr(\bG)\ :\ P_U^{\alpha}v\ne 0\right\}
\end{equation}
for the \emph{support} of $v$ with respect to the representation $U$, and $\supp v\to\infty$ for varying $v$ means that the support eventually misses every finite subset of $\Irr(\bG)$.

\begin{lemma}\label{le:u.rleg}
  A unitary representation $U\in M(\cC(\bG)\underline{\otimes}\cK(H))$ of a compact quantum group $\bG$ is uniform$_{\le 1}$ as a Banach-space representation provided
  \begin{equation*}
    \left(\supp v\to\infty \wedge \supp w\to\infty\right)
    \ 
    \xRightarrow{\quad}
    \ 
    (\id\otimes\omega_{v,w})U
    \xrightarrow[\quad\quad]{\quad\text{norm}\quad}
    0
  \end{equation*}
  for $(v,w)\in \left(H_{\le 1}\right)^2$ and $\omega_{v,w}:=\Braket{v\mid\bullet w}$.
\end{lemma}
\begin{proof}
  Per \Cref{le:u.lleg}, we have to verify \Cref{pr:cast2nm.sym.legs}'s condition \Cref{item:pr:cast2nm.sym.legs:unif.l}$_{\ell}$. A slightly more careful reworking of the implication \Cref{item:pr:cast2nm.sym.legs:unif.vect0.l}$_{r}$ $\Rightarrow$ \Cref{item:pr:cast2nm.sym.legs:unif.l}$_{\ell}$ (building on the proof of \cite[Proposition 2.2]{MR567834}) will deliver that outcome given the present hypotheses.

  Recall to that end that said proof begins by assuming a weak$^*$-0-convergent net $\left(\varphi_{\lambda}\right)_{\lambda}\subset \cC(\bG)^*$ with $U_{\varphi_{\lambda}}:=\left(\varphi_{\lambda}\otimes\id\right) U$ norm-bounded away from 0 and selecting corresponding
  \begin{equation*}
    \psi_{\lambda}:=\omega_{v_{\lambda},w_{\lambda}}\in \cL(H)_*
    ,\quad
    v_{\lambda},w_{\lambda}\in H_{\le 1}
  \end{equation*}
  with $\left|\psi_{\lambda} U_{\varphi_{\lambda}}\right|$ sufficiently close to $\left\|U_{\varphi_{\lambda}}\right\|$. Write
  \begin{equation*}
    P^{\cF}
    :=
    \sum_{\alpha\in \cF}P^{\alpha}
    ,\quad
    \cF\subseteq \Irr(\bG)
  \end{equation*}
  The hypothesis ensures that for finite $\cF\subseteq \Irr(\bG)$ sufficiently large,
  \begin{equation*}
    \begin{aligned}
      \left|\psi^{\cF}_{\lambda}U_{\varphi_{\lambda}}-\psi_{\lambda}U_{\varphi_{\lambda}}\right|
      &=
        \left|\varphi_{\lambda}\left(\id\otimes \psi^{\cF}_{\lambda}\right)U-\varphi_{\lambda}\left(\id\otimes \psi_{\lambda}\right)U\right|\\
      \psi^{\cF}_{\lambda}
      &\in
        \left\{
        \omega_{P^{\cF}v_{\lambda},w_{\lambda}}
        ,\
        \omega_{v_{\lambda},P^{\cF}w_{\lambda}}
        ,\
        \omega_{P^{\cF}v_{\lambda},P^{\cF}w_{\lambda}}
        \right\}
    \end{aligned}
  \end{equation*}
  is $\lambda$-uniformly small, so we may assume all $v_{\lambda}$ and $w_{\lambda}$ finitely supported on a single $\cF\subseteq \Irr(\bG)$. The truncation $\left(1\otimes P^{\cF}\right)U$ belonging to the algebraic tensor product
  \begin{equation*}
    \cC(\bG)^{\cF}\otimes \cL(H)
    ,\quad
    \cC(\bG):=\bigoplus_{\substack{\alpha\in \cF\\1\le i,j\le d_{\alpha}}}
    \bC u^{\alpha}_{ij},
  \end{equation*}
  the weak$^*$-norm Cauchy preservation assumed in the proof of \Cref{pr:cast2nm.sym.legs}'s \Cref{item:pr:cast2nm.sym.legs:unif.vect0.l}$_{r}$ $\Rightarrow$ \Cref{item:pr:cast2nm.sym.legs:unif.l}$_{\ell}$ implication holds for our $\psi_{\lambda}$ and the argument goes through.
\end{proof}

\begin{theorem}\label{th:rd.exp.dim}
  Let $\bG$ be a compact quantum group, $\cF\subseteq \Irr(\bG)$ a finite family and $\left(\alpha_n\right)_n\subset \Irr(\bG)$ with property RD with respect to $\ell_{\cF}$.

  If $\dim\alpha_n$ increase exponentially while the lengths $\ell\left(\alpha_n\right)$ increase polynomially, the unitary representation $\bigoplus_n \alpha_n$ is uniform$_{\le 1}$.
\end{theorem}
\begin{proof}
  We will apply \Cref{le:u.rleg}, proving that its hypothesis obtains under the present circumstances.

  Observe first that the individual positive operators $F_{\alpha_n}$ of \Cref{eq:falpha} are identities: \cite[Proposition 4.7]{MR2329000} proves this globally, assuming the quantum group as a whole has property RD, but the argument (relying on \cite[Lemmas 4.3 and 4.6]{MR2329000}) goes through in the local version needed here.
  
  Writing $H:=\bigoplus_n V_{\alpha_n}$ for the carrier Hilbert space of the direct-sum representation $U$ in question, we have to argue that for $v\in H$ with
  \begin{equation}\label{eq:suppv.away}
    \|v\|=1
    \ \wedge \ 
    \supp v\subseteq \left\{\alpha_m\ :\ m\gg 0\right\}
  \end{equation}
  ($\supp$ defined as in \Cref{eq:suppv}) the norm $\left\|(\id\otimes\omega_{v,v})U\right\|$ will be as small as desired provided only that $m\gg 0$ is indeed sufficiently large. The unimodularity condition $F_{\alpha}=1$ ensures \cite[(5.14), (5.15)]{wor} that the $u^{\alpha_n}_{ij}$ (are mutually orthogonal and) have $\|-\|_2$ norms $\frac 1{\sqrt{\dim\alpha_n}}$, hence the norm estimate
  \begin{equation*}
    \begin{aligned}
      \text{\Cref{eq:suppv.away}}
      \ 
      &\xRightarrow{\quad}
        \ 
        \exists\left(c>1\right)
        \left(
        \left\|(\id\otimes\omega_{v,v})U\right\|_2
        \le \max_{m:\alpha_m\in\supp v}\frac 1{\sqrt{\dim\alpha_m}}
        \le \frac 1{c^{m_0}}
        \right)\\
      &\xRightarrow{\quad\text{RD}\quad}
        \exists\left(\text{polynomial }P\right)
        \left(
        \left\|(\id\otimes\omega_{v,v})U\right\|
        \le
        \frac{P(m_0)}{c^{m_0}}
        \right)\\
      &\left(\text{with }m_0:=\min\left\{m\ :\ \alpha_m\in \supp v\right\}\right)
    \end{aligned}    
  \end{equation*}
  confirming the desired conclusion.
\end{proof}

\begin{example}\label{ex:oplus.uplus}
  Examples of compact quantum groups verifying the hypotheses of \Cref{th:rd.exp.dim} are well known and have been much studied (from various perspectives, RD-related or not): both
  \begin{itemize}[wide]
  \item the \emph{free orthogonal group} $\bG:=\bO^+(n)$, $n\ge 3$ (underlying function algebra $\cC(\bG)=A_o(I_n)$ in \cite[\S 4.3]{MR2329000}), with $\cO(\bG)$ freely generated by a unitary $n\times n$ matrix $u^{\alpha}=\left(u^{\alpha}_{ij}\right)$, $1\le i,j\le n$ of self-adjoint elements;

  \item and the \emph{free unitary group} $\bG:=\bU^+(n)$, $n\ge 3$, with $\cO(\bG)$ freely generated by unitary $u^{\alpha}\in \cO(\bG)\otimes M_n$ such that $\overline{u}^{\alpha}:=\left(u^{\alpha *}_{ij}\right)_{i,j}$ is also unitary
  \end{itemize}
  will do by \cite[Theorems 4.9 and 4.10]{MR2329000} respectively.
\end{example}

%%%%%%%%%%%%%%%%%%%%%%%%%%%%%%%%
%%%%%%%%%%%%%%%%%%%%%%%%%%%%%%%%

\addcontentsline{toc}{section}{References}
%\bibliography{bib}{}
%\bibliographystyle{plain}

% BEGIN INSERTED BBL (spectrum-compactness-q-xv3.bbl)

% END INSERTED BBL

\Addresses

\end{document}